\documentclass[fontsize=11pt,parskip=full]{scrartcl}
\usepackage{amsmath,amssymb,amsfonts,amsthm,mathabx}
\usepackage[dvips]{graphics}
\usepackage{enumerate}
\usepackage{mathrsfs}
\usepackage{subfigure}
\usepackage{color}

\usepackage[T1]{fontenc}

\DeclareSymbolFont{AMSb}{U}{msb}{m}{n}

\def\P{\mathbb P}
\def\N{\mathbb N}

\def\dd{\mathrm {d}}

\def\EE{\mathcal {E}}

\def\Fcal{\mathcal {F}}

\def\Ecal{\mathcal {E}}

\def\Mcal{\mathcal {M}}

\def\1{\,{\makebox[0pt][c]{\normalfont    1}
\makebox[2.5pt][c]{\raisebox{3.5pt}{\tiny {$\|$}}}
\makebox[-2.5pt][c]{\raisebox{1.7pt}{\tiny {$\|$}}}
\makebox[2.5pt][c]{} }}

\def\scirc{\mathbin{\raise.15ex\hbox{\scriptsize$\circ$}}}

\newtheorem{thm}{Theorem}[section]
\newtheorem{prop}[thm]{Proposition}
\newtheorem{cor}[thm]{Corollary}
\newtheorem{lemma}[thm]{Lemma}

\theoremstyle{definition}
\newtheorem{notation}[thm]{Notation}
\newtheorem{assumption}[thm]{Assumptions}
\newtheorem{properties}[thm]{Properties}
\newtheorem{defn}[thm]{Definition}
\newtheorem{remark}[thm]{Remark}

\newtheorem{ex}[thm]{Example}

\title{Convergence of Markov chain transition probabilities}

\author{Michael Scheutzow
  \thanks{Institut f\"ur Mathematik, MA 7-5, Fakult\"at II, 
    Technische Universit\"at Berlin, 
    Stra\ss e des 17.~Juni 136, 10623 Berlin, Germany;  \ 
    \small\texttt{ms{\scriptsize @}math.tu-berlin.de}}
\and Juni Schindler%
\thanks{Department of Mathematics, Imperial College London, South Kensington Campus, London SW7 2AZ, United Kingdom;  \ 
    \small\texttt{juni.schindler19{\scriptsize @}imperial.ac.uk}}}

\date{}

\begin{document}\maketitle

\begin{abstract}\noindent
  Consider a discrete time Markov chain with rather general state space which has an invariant probability measure $\mu$. There are several sufficient conditions in the literature
  which guarantee convergence of all or $\mu$-almost all transition probabilities to $\mu$ in the total variation (TV) metric: irreducibility plus aperiodicity, equivalence properties of
  transition probabilities, or coupling properties. In this work, we review and improve some of these criteria in such a way that they become necessary and sufficient for
  TV convergence of all respectively $\mu$-almost all transition probabilities. In addition, we discuss so-called generalized couplings.

\par\medskip

  \noindent\footnotesize
  \emph{2020 Mathematics Subject Classification} 
  Primary\, 60J05   % Random operators and equations
  \ Secondary\, 60G10 
%37B25 \ % Lyapunov functions and stability; attractors, repellers 
% 37B55 \ % Nonautonomous dynamical systems
% 37C70 \ % Attractors and repellers, topological structure 
% 37G35 \ % Attractors and their bifurcations 
%37H99 \ % Random dynamical systems: 
%        % `none of the above, but in this section' 
%37L55 \ % Infinite-dimensional random dynamical systems; 
        % stochastic equations
%          60D05 \ % Geometric probability, stochastic geometry, random sets
% 60G57 \ % Random measures
% 60H10 \ % Stochastic ordinary differential equations
% 60H15 \ % Stochastic partial differential equations 
% 60H25   % Random operators and equations
% 93E03
\end{abstract}


\section{Introduction}
It is a  classical result that all transition probabilities of a discrete time Markov chain with invariant probability measure (ipm) $\mu$ on a rather general state space $E$
converge to $\mu$ in the total variation metric provided that the chain is recurrent and aperiodic (\cite{MT09}). Further, {\em Doob's theorem} states that under appropriate
additional conditions, ultimate equivalence of every pair of transition probabilities implies the same result (see \cite[Theorem 4.2.1]{DZ96} or \cite{KS15}). Finally
the existence of {\em couplings} of chains starting at different initial conditions entails total variation convergence to $\mu$.  The goal of this paper is to
modify the sufficient conditions in the literature in such a way that they become equivalent. It will turn out, for example, that 
{\em asymptotic equivalence} of transition probabilities (which seems to be a new concept)  is equivalent to total variation convergence of all transition probabilities.
It is also of interest to find weaker conditions which only imply total
variation convergence of the transition probabilities starting from $\mu$ almost every $x \in E$. Again we will provide necessary and sufficient conditions similar to those
described above. We will also address a convergence property strictly between these two and again we will provide necessary and sufficient conditions. Apart from couplings
we will also formulate equivalent conditions in terms of {\em generalized couplings} for each of the convergence properties.

Throughout this paper $(E,\Ecal)$ denotes a measurable space for which $\Ecal$ is countably generated and the diagonal $\Delta:=\{(x,x): \,x \in E\}$ is in $\Ecal \otimes \Ecal$
(or, equivalently, $\Ecal$ is countably generated and separates points or, equivalently, $\Ecal$ is countably generated and all singletons are in $\Ecal$ (see \cite[p. 116]{E11}).
Let $P$ be a Markov kernel on $E$ and denote the corresponding $n$-step transition probability by $P_n(.,.)$, $n \in \N_0$.
$\P_x$  denotes the law of the Markov chain starting at $x \in E$. Note that $\P_x$ is a probability measure on $(E^{\N_0},\EE^{\otimes \N_0})$. We will often identify a Markov chain and
its Markov kernel $P$ and denote the corresponding Markov chain by $X$. We denote the {\em total variation} metric on the space of probability measures on $(E,\Ecal)$ by $d$,
i.e.~$d(\nu_1,\nu_2):=\sup_{A \in \Ecal}|\nu_1(A)-\nu_2(A)|$.  We say that $P_n(x,.)$ {\em converges} to a probability measure $\mu$ on $(E,\EE)$ if $P_n(x,.)$ converges to $\mu$
in the total variation metric as $n \to \infty$. Throughout  the paper we will assume that $P$ admits an ipm  $\mu$ (but we will not assume
uniqueness of $\mu$). From now on, the letter $\mu$ will always denote an invariant probability measure of the Markov chain $X$ associated to $P$.  

Let $\nu_1$ and $\nu_2$ be measures on the same measurable space  $(\bar E,\bar \EE)$. Then
we say (as usual) that $\nu_1$ is absolutely continuous with respect to $\nu_2$ (notation $\nu_1 \ll \nu_2$) if $A\in\bar\EE$ with $\nu_2(A)=0$ implies $\nu_1(A)=0$, and that $\nu_1$ and $\nu_2$ are   equivalent (denoted $\nu_1 \sim \nu_2$)
if they are mutually absolutely continuous. Further we write $\nu_1\not\perp\nu_2 $ if $\nu_1$ and $\nu_2$ are non-singular, i.e. there does not exist a set $A \in \bar \Ecal$
such that $\nu_1(A)=0$ and $\nu_2(A^c)=0$.   Any measure $\xi$ on $(\bar E\times \bar E,\bar \EE \otimes \bar \EE)$
with marginals $\nu_1$ and $\nu_2$ is called a {\em coupling} of $\nu_1$ and $\nu_2$. We write $\xi \in C(\nu_1,\nu_2)$. Recall the {\em coupling equality}: for
probability measures $\nu_1$ and $\nu_2$ on $(\bar E, \bar\EE)$, we have $d(\nu_1,\nu_2)=\inf\{\xi(\Delta):\,\xi \in C(\nu_1,\nu_2)\}$ (\cite[Theorem 2.2.2]{K18}). 
We will call a pair $(X,Y)$ of $\bar E$-valued random variables defined on the same probability space a coupling of the probability measures
$\nu_1$ and $\nu_2$ on $(\bar E,\bar \EE)$, if their joint law is a coupling of $\nu_1$ and $\nu_2$. Below we will deal with the cases $\bar E:=E$ and $\bar E:=E^{\N_0}$.  We will define the concept of a {\em generalized coupling} later. Generalized (asymptotic) couplings are particularly useful
to prove {\em weak} convergence of transition probabilities (see \cite{KS18} and \cite{BKS19}) but (non-asymptotic) generalized couplings can also be used
to establish upper bounds on the total variation distance of transition probabilities (see \cite[Proof of Theorem 1.1]{ERS09}).  

We will formulate all results in the discrete-time set-up. This is essentially without loss of generality. Indeed, assume that $\mu$ is an invariant probability measure of an $E$-valued continuous-time Markov process. Then $\mu$ is also an ipm of the associated {\em skeleton chain} sampled at times
$0,h,2h,...$ and for each $x \in E$ total variation convergence of $P_{nh}(x,.)$ to $\mu$  (as $n \to \infty$) for some $h>0$ is equivalent to  total variation convergence of $P_t(x,.)$ to $\mu$ since $t \mapsto d\big(P_t(x,.),\mu\big)$ is non-increasing.

Once one has established convergence of all or almost all transition probabilities then it is natural to ask for the speed of convergence. A large
number of papers have been devoted to these questions, for example \cite{HM11}, \cite{RR04} and \cite{K18}. We will however, not touch these questions here. 

At some point we will need a stronger condition on the  measurable space $(E,\Ecal)$: as usual, we say that $(E,\Ecal)$ is a {\em Borel} space if it is isomorphic (as a measurable space)
to a Borel subset of $[0,1]$. In particular, this holds for a complete, separable  metric space $E$ equipped with its Borel $\sigma$-field  $\Ecal$. 

\section{Necessary and sufficient conditions for total variation convergence}

Let $\big(X_n\big)_{n \in \N_0}$ be a Markov chain with transition kernel $P$, ipm $\mu$ and state space $(E,\Ecal)$ as in the introduction. We adopt the following notation (cf.~\cite{MT09}).

\begin{notation} For $x \in E$, $A \in \Ecal$,
  \begin{align*}
    Q(x,A):=&\P_x\big(\big\{X_n \in A\mbox{ for infinitely many }n\in \N\big\}\big),\\   
    L(x,A):=&\P_x\Big(\bigcup_{n=1}^\infty \big\{X_n \in A\big\}\Big).
  \end{align*}
\end{notation}

 We start by defining three properties of increasing generality which we will be interested in.

\begin{properties} We say that
  \begin{itemize}
    \item Property P$_1$ holds if $P_n(x,.)$ converges to $\mu$ for every $x \in E$.
    \item Property P$_2$ holds if $P_n(x,.)$ converges to $\mu$ for $\mu$-almost all $x \in E$ and\\ $\lim_{n \to \infty}d(P_n(x,.),\mu)<1$ for all $x \in E$.
    \item Property P$_3$ holds if $P_n(x,.)$ converges to $\mu$ for $\mu$-almost all $x \in E$.
  \end{itemize}
\end{properties}    

\begin{remark}\label{12}
  Note that Properties P$_1$ and P$_2$ both  imply uniqueness of $\mu$ (we will show the latter claim in Remark \ref{uniq}). Note also that $\lim_{n \to \infty}d(P_n(x,.),\mu)$ always exists since $\mu$ is invariant and the
  total variation distance can never increase when applying a measurable map. Therefore, we could replace ``$\lim_{n \to \infty}d(P_n(x,.),\mu)<1$ for all $x \in E$'' in P$_2$ by
  ``for each $x$ there exists some $n \in \N_0$ such that $d(P_n(x,.),\mu)<1$'' without changing the class of chains for which P$_2$ holds.
  One might also be interested in a modification $\tilde {\mathrm{P}}_2$ of Property P$_2$ in which the last property
  $\lim_{n \to \infty}d(P_n(x,.),\mu)<1$ for all $x \in E$ is replaced by uniqueness of $\mu$. Clearly, P$_2$ is stronger than $\tilde {\mathrm{P}}_2$  and it is easy to see that it is strictly stronger.
  Property $\tilde {\mathrm{P}}_2$ was studied in \cite{KS15}, for example, but P$_2$ is more closely related to conditions studied in the literature. We will see, in particular,
  that the assumptions of \cite[Corollary  1]{KS15} do not only imply $\tilde {\mathrm{P}}_2$ but even   P$_2$.
  Example \ref{peri} shows that one cannot delete the first part of property P$_2$ without changing the class of chains for which it holds.
\end{remark}
 We will  define four sets of assumptions, one in terms of equivalence or non-singularity of transition probabilities, one in terms of
  aperiodicity and recurrence or irreducibility properties, one in terms of couplings and one in terms of generalized couplings. It will turn out that
  all assumption with index $i$, $i \in \{1,2,3\}$, not only imply property P$_i$ but are also {\em necessary} for P$_i$ to hold. In some cases we
  formulate conditions with an additional prime (or some other symbol) which will formally be stronger than the same condition without prime but which will in fact turn out to be
  equivalent (at least when the state space is Borel). 
  Before we state various assumptions we define the (possibly new) concept of asymptotic equivalence of transition probabilities.

\begin{defn}
    We say that the states $x\in E$ and $y\in E$ are {\em asymptotically equivalent} if for each $\varepsilon>0$ there exists some $n \in \N$ and a set $A \in \Ecal$ such that
    $P_n(x,A)\ge 1-\varepsilon$, $P_n(y,A)\ge 1-\varepsilon$, and the measures $P_n(x,.)$ and $P_n(y,.)$ restricted to the set $A$ are equivalent.
  \end{defn}

\begin{remark}\label{equirel}
  Note that if for given $x,y \in E$, $\varepsilon>0$ and $n \in \N$ there exists a set $A$ as in the previous definition, then there exists a set $\bar A$ as in the
  previous definition (with the same $\varepsilon$) if $n$ is replaced by $n+1$ (and, by iteration, the same holds for all integers larger than $n$).
  This implies, in particular, that asymptotic equivalence induces an equivalence relation on $E$. %(we will use this fact later).
\end{remark}
  
 \begin{assumption} We say that
  \begin{itemize}
    \item Assumption A$_1$ holds if all pairs $(x,y)\in E \times E$ are asymptotically equivalent.
    \item Assumption A$_2$ holds if for all $(x,y)\in E \times E$ there exists some $n=n_{x,y}\in\N$ such that $P_n(x,.)\not\perp P_n(y,.)$.
    \item Assumption A$_3$ holds if for $\mu \otimes \mu$-almost all $(x,y)\in E \times E$ there exists some $n=n_{x,y}\in\N$ such that $P_n(x,.)\not\perp P_n(y,.)$.
    \item Assumption A$_3'$ holds if $\mu \otimes \mu$-almost all $(x,y)\in E \times E$ are asymptotically equivalent.
  \end{itemize}
\end{assumption}  

Lemma \ref{mbae} states that the set of all $(x,y)\in E \times E$ which are asymptotically equivalent is a measurable subset of $(E \times E,\Ecal \otimes \Ecal)$.

\begin{remark}\label{obvious}
  Obviously, Property P$_1$ implies that any two states $x,y$ are asymptotically equivalent (i.e.~A$_1$ holds)  while the simple Example \ref{simple} shows that
  it does {\em not} imply the stronger property ``for all $x,y \in E$ there exists some $n=n_{x,y} \in \N_0$ such that $P_n(x,.)\sim P_n(y,.)$'' under which P$_1$
  was shown in \cite[Theorem 1]{KS15}. 
  \end{remark}

Before we state the second set of assumptions, we define the concepts of aperiodicity, irreducibility and the Harris property for a Markov kernel $P$ with invariant measure $\mu$. 

\begin{defn}\label{aper} \cite[p. 32]{RR04}
  The  Markov kernel $P$ (with invariant probability measure  $\mu$) is called \emph{$d$-periodic},
  if $d\ge2$, and  there are disjoint sets $E_{1}, E_{2},...,E_{d}\in\mathcal{E}$ with  $\mu(E_{1})>0$ that fulfill
\begin{equation}\label{aperi}
	P(x,E_{i+1 (\mathrm{mod}\, d)})=1 \quad \forall x\in E_{i}, 1\le i\le d.
\end{equation}
The chain is called \emph{aperiodic} if no such $d\ge2$ exists.
\end{defn}

\begin{defn}\label{irreducible}
  The Markov kernel $P$  is called {\em $\phi$-irreducible} if $\phi$ is a non-trivial $\sigma$-finite measure on $(E,\EE)$ such that for all $A\in\EE$ with $\phi(A)>0$
  and all $x\in E$ we have $L(x,A)>0$ (or, equivalently, there exists some $n=n(x,A)\in\mathbb{N}$ such that $P_{n}(x,A)>0$).
  $P$ is called {\em irreducible} if $P$ is $\phi$-irreducible for some non-trivial $\phi$.
  We say that $P$ is {\em weakly irreducible}
  (with respect to the  given ipm $\mu$) if  there exists some  non-trivial  $\sigma$-finite measure $\phi$ on $(E,\EE)$ and a set $E_0\in \EE$ satisfying
  $\mu(E_0)=1$ such that for every $x \in E_0$ and every $A\in\EE$ with $\phi(A)>0$ we have $L(x,A)>0$.    
\end{defn}

\begin{remark}\label{rem:irr}
  It is straightforward to check that if $\phi$ is as in the definition (either part), then $\phi \ll \mu$. Further, if $P$ is (weakly) $\mu$-irreducible then $P$ is (weakly)
  $\phi$-irreducible for every non-trivial $\sigma$-finite measure on $(E,\EE)$ satisfying $\phi \ll \mu$. We will show in Proposition \ref{irredu} the less obvious fact
  that ($\phi$-)irreducibility implies $\mu$-irreducibility (which, in the terminology  of \cite[Proposition 4.2.2]{MT09}, means that $\mu$ is the {\em maximal irreducibility measure}).
  We will use Proposition \ref{irredu} only in the proof of Theorem \ref{second}.
\end{remark}

\begin{defn} \cite[p. 199]{MT09}
  $P$ or the associated Markov chain $X$ are called \emph{Harris} (or \emph{Harris recurrent}), if there exists a non-trivial %if $P$ is irreducible with respect to some non-zero
  $\sigma$-finite measure $\phi$ on $(E,\EE)$ such that for all $A\in\EE$ with $\phi(A)>0$ and all
  $x\in E$ we have $Q(x,A)=1$ (or, equivalently, $L(x,A)=1$ for all $x\in E$ and $A \in \EE$
  with $\phi(A)>0$).
\end{defn}

\begin{assumption} We say that
  \begin{itemize} 
    \item Assumption B$_1$ holds if $P$ is aperiodic and Harris.
    \item Assumption B$_2$ holds if $P$ is aperiodic and irreducible.
    \item Assumption B$_3$ holds if $P$ is aperiodic and weakly irreducible.
  \end{itemize}
\end{assumption}    

Note that Harris recurrence implies irreducibility, so  B$_1$ implies  B$_2$. 

Let $\Mcal(\bar E)$ be the set of all probability measures on the measurable space $(\bar E,\bar \EE)$.
For $\xi \in \Mcal (\bar E \times \bar E)$, we denote the
$i$-th marginal by $\xi^i$, $i \in \{1,2\}$. If $(\bar E,\bar \EE)=(E^{\N_0},\bar \EE^{\N_0})$, then  we denote the projection of $\xi$ resp.~$\xi^i$ onto the $k$-th coordinate by $\xi_k$ resp.~$\xi^i_k$, $k \in \N_0$, $i \in \{1,2\}$.
\begin{assumption} We say that
  \begin{itemize}
  \item Assumption C$_1$ holds if for each $x,y \in E$ and $m \in \N$ there exists some $k_m\in \N_0$ and a coupling $\zeta [m]\in C(P_{k_m}(x,.),P_{k_m}(y,.))$ such that
    $\zeta[m](\Delta)\ge 1-\frac 1m$.

  \item Assumption $\mathrm{\hat C_1}$ holds if for each $x,y \in E$ and $m \in \N$ there exists  a coupling $\zeta [m]\in C(P_{m}(x,.),P_{m}(y,.))$ such that
    $\lim_{m \to \infty}\zeta [m](\Delta)=1$.

  \item Assumption $\mathrm{\mathring C_1}$ holds if for each $x,y \in E$ there exists  a coupling $\xi\in C(\P_x,\P_y)$ such that
    $\lim_{m \to \infty}\xi_m(\Delta)=1$.
  
  \item Assumption C$_1'$  holds if  for each $x,y \in E$ there exists a coupling $(X_k)_{k \in \N_0}$,  $(Y_k)_{k \in \N_0}$
    of $\P_x$ and $\P_y$ on some space
    $(\Omega,\Fcal,\P)$ such that $\lim_{n \to \infty}\P\big( X_k=Y_k \mbox{ for all } k \ge n\big)=1$.

  \item Assumption C$_2$ holds if for all $x,y \in E$ there exists some $k \in \N_0$ and a coupling $\zeta \in C(P_k(x,.),P_k(y,.))$ such that
    $\zeta(\Delta)>0$.
   
  \item Assumption C$_2'$ holds if for each $x,y \in E$ there exists a coupling $(X_k)_{k \in \N_0}$,  $(Y_k)_{k \in \N_0}$ of $\P_x$ and $\P_y$ on some
     space $(\Omega,\Fcal,\P)$ such that $\liminf_{n \to \infty}\P\big( X_k=Y_k \mbox{ for all } k \ge n\big)>0$ and for $\mu\otimes\mu$-almost every $(x,y)\in E\times E$ there exists
     a coupling $(X_k)_{k \in \N_0}$,  $(Y_k)_{k \in \N_0}$ of $\P_x$ and $\P_y$ on some
     space $(\Omega,\Fcal,\P)$ such that $\lim_{n \to \infty}\P\big( X_k=Y_k \mbox{ for all } k \ge n\big)=1$.

   \item Assumption C$_3$ holds if for $\mu\otimes\mu$-almost every $(x,y)\in E\times E$ there exists some $k \in \N_0$ and a coupling $\zeta\in C(P_k(x,.),P_k(y,.))$ such that
    $\zeta(\Delta)>0$.

  \item Assumption C$_3'$ holds if for $\mu\otimes\mu$-almost every $(x,y)\in E\times E$ there exists a coupling $(X_k)_{k \in \N_0}$,  $(Y_k)_{k \in \N_0}$ of $\P_x$ and $\P_y$ on some
    space $(\Omega,\Fcal,\P)$ such that $\lim_{n \to \infty}\P\big( X_k=Y_k \mbox{ for all } k \ge n\big)=1$.    
  \end{itemize}
\end{assumption}

We chose Condition C$_i$ such that it is as weak as possible and  C$_i'$ such that it is as strong as possible subject to the requirement that both are equivalent to
all other conditions with the same index $i$ (in case the state space is Borel). 
Note that there are several natural conditions in between C$_i$ and C$_i'$ ($i=1,2,3$) for which there is no need to state them, since they will all turn out to be equivalent
(at least in the Borel case). 
Finally, we define the concept of a {\em generalized coupling}.

\begin{defn} For probability measures $\nu_1$ and $\nu_2$ on $(\bar E,\bar \EE)$, define
  \begin{itemize}
    \item $\tilde C(\nu_1,\nu_2):=\big\{\xi \in \Mcal(\bar E \times \bar E):\,\xi^1 \ll \nu_1,\,\xi^2 \ll \nu_2\big\}$,
    \item $\check C(\nu_1,\nu_2):=\big\{\xi \in \Mcal(\bar E \times \bar E):\,\xi^1 \ll \nu_1,\,\xi^2 \sim \nu_2\big\}$.
  \end{itemize}
\end{defn}

\begin{assumption} We say that
  \begin{itemize} 
    \item Assumption G$_1$ holds if  for each pair $(x,y) \in E\times E$ there exists some $\xi \in \check C(\P_x,\P_y)$ such that
    $\lim_{k \to \infty}\xi_k(\Delta)=1$,
  \item Assumption G$_2$ holds if for each pair $(x,y) \in E\times E$ there exists some $k \in \N$ and
    $\zeta\in \tilde C(P_k(x,.),P_k(y,.))$ such that $\zeta(\Delta)>0$.
    %$\xi \in \tilde C(\P_x,\P_y)$  and some $k \in \N$ such that
    %$\xi_k(\Delta)>0$,
  \item Assumption G$_3$ holds if for $\mu \otimes \mu$-almost every $(x,y)\in E\times E$ there exists some $k \in \N$ and
    $\zeta\in \tilde C(P_k(x,.),P_k(y,.))$ such that $\zeta(\Delta)>0$. 
  \end{itemize}
\end{assumption}

If we change ``$>0$'' in G$_2$ to ``$=1$'', then the resulting condition is {\em not} equivalent to G$_1$ (see Example \ref{standard}).

\begin{thm}\label{first}
  $\mathrm{A}_1$, $\mathrm{B}_1$,  $\mathrm{C}_1$, $\mathrm{\hat C}_1$, and  $\mathrm{P}_1$ are equivalent and
  $\mathrm{C}_1'\Rightarrow \mathrm{\mathring C_1}\Rightarrow   \mathrm{G}_1\Rightarrow \mathrm{A}_1$.  
  If $(E,\Ecal)$ is  Borel, then all these conditions are equivalent. 
\end{thm} 

\begin{thm}\label{second}
$\mathrm{A}_2$, $\mathrm{B}_2$,  $\mathrm{C}_2$, $\mathrm{G}_2$, and  $\mathrm{P}_2$ are equivalent and are implied by $\mathrm{C}_2'$. If $(E,\Ecal)$ is Borel,
  then each of the equivalent conditions implies $\mathrm{C}_2'$. 
\end{thm}

\begin{thm}\label{third}
$\mathrm{A}_3$, $\mathrm{A}_3'$, $\mathrm{B}_3$,  $\mathrm{C}_3$, $\mathrm{G}_3$, and  $\mathrm{P}_3$ are equivalent and are implied by  $\mathrm{C}_3'$. If $(E,\Ecal)$ is Borel,
  then each of the equivalent conditions implies $\mathrm{C}_3'$. 
\end{thm}  

\begin{remark}
We do not know if the equivalence of all conditions with the same index holds even under our general conditions on the space $(E,\Ecal)$. We will comment on this in Remark \ref{allgemein}.
\end{remark}

\section{First results and the proof of Theorem \ref{first}}
Let us first state those implications in the theorems which are obvious from the definitions or are well-known.
\begin{prop}\label{easy} We have
  \begin{itemize}  
    \item [a)] $\mathrm{B}_1\Rightarrow\mathrm{P}_1$,
    \item [b)] $\mathrm{C}_1'\Rightarrow \mathrm{\mathring C_1}\Rightarrow   \mathrm{G}_1,\quad \mathrm{P}_1\Rightarrow \mathrm{\hat C}_1\Rightarrow \mathrm{C}_1\Rightarrow   \mathrm{A}_1$,
    \item [c)] $\mathrm{C}_2'\Rightarrow \mathrm{C}_2\Rightarrow   \mathrm{G}_2\Rightarrow \mathrm{A}_2,\quad \mathrm{P}_2 \Rightarrow   \mathrm{A}_2\Leftrightarrow \mathrm{C}_2$,
    \item [d)] $\mathrm{P}_3\Rightarrow  \mathrm{A}'_3  \Rightarrow\mathrm{A}_3, \quad  \mathrm{C}'_3\Rightarrow\mathrm{C}_3 \Leftrightarrow\mathrm{A}_3$, and 
$\mathrm{C}_3\Rightarrow\mathrm{G}_3\Rightarrow\mathrm{A}_3$.
  \end{itemize}
\end{prop}

\begin{proof}
  Statement a) is a classical result and a proof can be found for example in \cite[p.~328]{MT09}.
  The remaining implications are either obvious or easy consequences of the coupling equality stated in the introduction.
\end{proof}

We continue by providing  a slightly generalized version of the {\em Recurrence Lemma} from \cite[Lemma 2]{KS15} that will turn out to be 
  useful later.

\begin{lemma}[Recurrence Lemma]\label{recur}
  Assume that $P$ satisfies Assumption $\mathrm{A}_3$. %or $P$ is weakly irreducible.
  Then for any $B\in\EE$ with $\mu(B)>0$, for $\mu$-almost every $x\in E$
  \begin{equation}\label{recureq}
    Q(x,B)= 1.
  \end{equation}
  If, moreover, $P$ satisfies Assumption $\mathrm{A}_2$, %or $P$ is irreducible,
  then
$$%\begin{equation}\label{recureq2}
Q(x,B)>0
$$%\end{equation}
holds for {\em every} $x \in E$.

If, moreover, $P$ satisfies Assumption $\mathrm{A}_1$, then \eqref{recureq} holds for {\em every} $x \in E$.
\end{lemma}

\begin{proof}
For $B\in \Ecal$ with $\mu(B)>0$ define $\psi(x):=Q(x,B)=\P_x(X_k\in B$ infinitely often), $x \in E$. Starting $X_0$ with law $\mu$, we see that
$\psi(X_n)$, $n \in \N_0$ is both stationary and a (bounded) martingale which converges to $\1_{\{X_k\in B \mbox{ i.o.}\}}$ almost surely which implies
$\psi(x)\in\{0,1\}$ for $\mu$-almost all $x \in E$. Let $\Psi_i=\{x:\,\psi(x)=i\}$, $i \in \{0,1\}$. Then, by the martingale property,
$P_n(x,\Psi_i)=1$ for all $n \in \N_0$ and for $\mu$-almost all $x \in \Psi_i$, $i \in \{0,1\}$. If A$_3$ holds, %or $P$ is weakly irreducible,
then (at least) one of the sets
$\Psi_0,\,\Psi_1$ has $\mu$-measure zero.
Since $\mu(B)>0$, Birkhoff's ergodic theorem implies $\mu\big(\Psi_1\big)>0$, so $\mu\big(\Psi_0\big)=0$ and $\mu\big(\Psi_1\big)=1$, finishing the proof of the first statement.

Let Assumption A$_2$ hold and fix $x \in E$. Since $P_n(y,\Psi_1)=1$ for $\mu$-almost all $y$ and all $n\in \N_0$, there exists some $y_0\in E$ such that  $P_n(y_0,\Psi_1)=1$
for all $n \in \N_0$. Now A$_2$ applied to $x$ and $y_0$ shows that there exists some $n \in \N$ such that $P_n(x,\Psi_1)>0$, finishing the proof of the second claim.

Let Assumption A$_1$ hold and fix $x \in E$.  As above, there exists some $y_0\in E$ such that  $P_n(y_0,\Psi_1)=1$
for all $n \in \N_0$. Now A$_1$ applied to $x$ and $y_0$ shows that $\lim_{n \to \infty}P_n(x,\Psi_1)=1$, so $x \in \Psi_1$ and therefore \eqref{recureq} holds. 
\end{proof}
\begin{prop}\label{AperProp}
A Markov kernel $P$ which satisfies Assumption $\mathrm{A}_{3}$ is aperiodic.
\end{prop}

\begin{proof}
Suppose $P$ has period $d\ge2$, and let  $E_{1}, E_{2},...,E_{d}$ be as in Definition \ref{aper}. Then $\mu(E_{i})>0$ for $i=1,2,...,d$.
Choose $x\in E_1$, $y\in E_2$, and $n\in\N$ arbitrarily. Then $P_{n}(x,E_{n+1 (\mathrm{mod}\, d)})=1$ and $P_{n}(y,E_{n+2 (\mathrm{mod}\, d)})=1$ and therefore $P_{n}(x,\cdot)\perp P_{n}(y,\cdot)$.
% Thus, for all $(x,y)\in E_{d-1}\times E_{d}$ and $n\in\mathbb{N}$ there holds $P_{n}(x,\cdot)\perp P_{n}(y,\cdot)$.
This contradicts Assumption A$_{3}$ since $(\mu\otimes\mu)( E_1\times E_2)>0$.
\end{proof}

\begin{cor}\label{coro}
$\mathrm{A}_1 \Rightarrow \mathrm{B}_1$,  $\mathrm{A}_2 \Rightarrow \mathrm{B_2}$, and  $\mathrm{A}_3' \Rightarrow \mathrm{B_3}$.
\end{cor}

\begin{proof}
  Lemma \ref{recur}, Proposition \ref{AperProp} and Remark \ref{rem:irr} immediately imply the first two implications (with $\phi:=\mu$) but not the last one since the conclusion of the Recurrence Lemma under the assumption
  $\mathrm{A}_3$ (or the stronger assumption $\mathrm{A}_3' $) is weaker than weak irreducibility (the exceptional sets of $\mu$-measure 0 may depend on the set $B$ and there may be uncountably many such
  sets). Therefore, we argue as follows: for $x \in E$, let $R_x:=\{y \in E: x,y \mbox{ are as.~equiv.}\}$. Assumption $\mathrm{A}_3'$ and Lemma \ref{mbae} imply that $R_x \in \Ecal$ and $\mu(R_x)=1$
    for $\mu$-a.a.~$x \in E$. Fix $x \in E$ such that $\mu(R_x)=1$. Since asymptotic equivalence is an equivalence relation by Remark \ref{equirel}, it follows that property $ \mathrm{A}_1$ holds on
    $R_x$. Using Lemma \ref{restrict}, we see that  $\mathrm{B}_1$ holds on $R_x$ and hence $\mathrm{B}_3$ holds on $E$.
\end{proof}

Before we step into the proofs of Theorems \ref{first}, \ref{second}, and \ref{third}, we sketch how one can see that $\mathrm{A}_i$ implies $\mathrm{C}_i'$ for
$i \in \{1,2,3\}$. The proofs are largely identical to those in \cite{KS15} where the implications $\mathrm{\tilde A_1} \Rightarrow$ P$_1$,  A$_2 \Rightarrow \mathrm{\tilde P}_2$, and  A$_3 \Rightarrow$ P$_3$ were
shown (with $\mathrm{\tilde A_1}$ slightly stronger than $\mathrm{A_1}$ and  $\mathrm{\tilde P_2}$ slightly weaker than $\mathrm{P_2}$  and  without the assumptions that the state space is Borel). We will need the  Borel property only at the end of the proof when we apply the gluing lemma.

\begin{prop}\label{atta} We have
  $$
\mathrm{A}_3 \Rightarrow \mathrm{P}_3.
  $$
Further,  if $(E,\Ecal)$ is  Borel, then
  $$
  \mathrm{A}_1 \Rightarrow \mathrm{C}_1',  \mathrm{A}_2 \Rightarrow \mathrm{C}_2', \mbox{ and } \mathrm{A}_3 \Rightarrow \mathrm{C}_3'.
  $$
\end{prop}

\begin{proof}[Idea of the proof]
  Under $\mathrm{A}_3$,  we define for $N \in \N$ and $p \in (0,1)$
  $$
C_{N,p}:=\big\{(x,y) \in E \times E: \,d\big( P_N(x,.),P_N(y,.)\big)\le 1-p\big\}.
  $$
  $C_{N,p}\in \Ecal\otimes \Ecal$ by  Proposition \ref{mbTV}   and   Assumption $\mathrm{A}_3$ implies $\mu\otimes \mu(C_{N,p})>0$ for some $N$ and $p$.
  Fix  $N$ and $p$ and write $C:=C_{N,p}$.
  Let us first assume that $N=1$ (this is without loss of generality for proving $\mathrm{A}_3 \Rightarrow \mathrm{P}_3$ but not 
  without loss of generality for proving $\mathrm{A}_3 \Rightarrow\mathrm{C}_3'$). In \cite{KS15}, the authors proceed by constructing a Markov chain
  $Z_n$, $n \in \N_0$ on the product space $E \times E$, which is a coupling of two chains with Markov kernel $P$ with transition kernel $S$ defined as
  $$
  S\big((x,y),.\big):=\left\{
    \begin{array}{ll}
      Q\big((x,y),.\big)&\mbox{if } (x,y)\in C\\
      R\big((x,y),.\big)&\mbox{otherwise.}
    \end{array}
    \right.
  $$
  Here, $R\big((x,y),.\big)$ is the product of $P\big(x,.\big)$ and  $P\big(y,.\big)$ and the kernel $Q$ satisfies $Q\big((x,y),\Delta\big)=1-d\big(P(x,.),P(y,.)\big)$
  and $Q\big((x,y),.\big)$ restricted to $(E \times E)\backslash \Delta$ is absolutely continuous with repect to the product of $P\big(x,.\big)$ and  $P\big(y,.\big)$
  (the fact that such a kernel $Q$ exists is stated in \cite[Lemma 1]{KS15}). The idea behind the definition of the kernel $S$ is the following: whenever
  the chain on $E \times E$ is in a state $(x,y)\in C$, then we try to couple the two coordinates in the next step by applying $Q$ which maximizes the coupling
  probability. Otherwise, we let the two coordinates move independently until the pair hits the set $C$. As soon as the chain $Z$ hits the diagonal $\Delta$ it
  remains in that state forever. It remains to ensure that the set $C$ is hit infinitely many times and therefore the process $Z_n$ will almost surely eventually hit
  $\Delta$. The fact that $(Z_n)$ will hit the set $C$ almost surely in finite time can be seen as follows: consider an {\em independent} coupling $(W_n)$ of two copies of the chain.
  Since $\mu\otimes \mu(C)>0$, the Recurrence Lemma shows that $(W_n)$ will hit the set $C$ almost surely in finite time for almost all initial conditions and even for all
  initial conditions if we assume A$_1$. Since, up to the first hitting time of the set $C$,
  the processes $W$ and $Z$ have the same law, $(Z_n)$ will also hit the set $C$ almost surely in finite time. If the coupling attempt at that time is unsuccessful, then the
  chain $Z$ again performs an independent coupling up to the next hit of $C$, which, by the same argument (and the strong Markov property and the assumptions on the kernel $Q$),
  is an almost sure event. The constructed coupling therefore shows that C$_1'$ holds under A$_1$ and both C$_3'$ and $\mathrm{P}_3$ hold under A$_3$. Further, under A$_2$, for any pair $x,y \in E$ the
  probability that the constructed coupling is successful, is strictly positive by the second part of the Recurrence Lemma, so C$_2'$ holds. This proves the claims in case $N$ in the definition of
  the set $C_{N,p}$ can be chosen to be 1.

  Finally, we assume that $N\ge 2$. The first claim follows from the case $N=1$ since $n \mapsto d(P_n(x,.),\mu)$ is non-increasing. To see the remaining claims, we apply the previous
  consideration to the skeleton chain evaluated at integer multiples of $N$ and  obtain corresponding couplings $Z_{nN}=(X_{nN},Y_{nN})$,
  $n \in \N_0$ for the skeleton chains as above. We have to make sure that these can be appropriately interpolated between successive multiples of $N$. This follows from an application of
  the gluing lemma in the appendix (which requires the state space to be  Borel) to each gap between successive multiples of $N$ (with conditionally independent interpolations), see
  \cite[p.43]{RR04} for a similar construction (it seems that the authors forgot to mention that this construction requires the space to be  Borel, see Remark \ref{allgemein}).
\end{proof}

\begin{proof}[Proof of Theorem \ref{first}] Observing Proposition \ref{easy}, Corollary \ref{coro} and Proposition \ref{atta} the claim follows once we prove that $\mathrm{G}_1\Rightarrow\mathrm{A}_1$.

$\mathrm{G}_1\Rightarrow\mathrm{A}_1$:  Fix a pair $(x,y)\in E\times E$. We show that $x$ and $y$ are asymptotically equivalent. Fix %$\xi \in \check C(\P_x,\P_y)$ and
  $\varepsilon>0$. By assumption there exists some $\xi \in \check C(\P_x,\P_y)$  such that $\lim_{k \to \infty}\xi_k(\Delta)=1$.
  Since $\xi^2$ and $\P_y$ are equivalent, we can find some $\delta>0$ such that for every  $\Gamma\in\EE^{\otimes \N_0}$ satisfying   $\xi^2(\Gamma)<\delta$, we have
  $\P_y(\Gamma)<\varepsilon$. Let $n_0\in \N_0$ be such that $\xi_k(\Delta)> 1-\delta$  for every $k\ge n_0$.   Then, for $B \in \Ecal$
  and $n \ge n_0$,
  $$
P_n(x,B)=0\; \Rightarrow \xi_n^1(B)=0\; \Rightarrow \xi_n^2(B)<\delta\; \Rightarrow P_n(y,B)<\varepsilon,
$$
where we used absolute continuity of $\xi_n^1$ with respect to $P_n(x,.)$ in the first step.
Reversing the roles of $x$ and $y$ we get $P_n(y,B)=0  \Rightarrow P_n(x,B)<\varepsilon$ for all $n\ge n_1$. Fix $n \ge n_0\vee n_1$ and let
$B_0 \in \EE$  be a set which maximizes  $P_n(y,B)$ among all sets $B \in \EE$ which satisfy $P_n(x,B)=0$ and let $C_0\in \EE$  be a set which maximizes
$P_n(x,C)$ among all sets $C \in \EE$ which satisfy $P_n(y,C)=0$. Define $A:=E\backslash (B_0 \cup C_0)$. Then $P_n(x,A)\ge 1-\varepsilon$, $P_n(y,A)\ge 1-\varepsilon$
and the restrictions of $P_n(x,.)$ and $P_n(x,.)$ to $A$ are equivalent. The claim follows since $\varepsilon>0$ was arbitrary.
\end{proof}

\section{Proofs of Theorems \ref{second} and \ref{third}}

\begin{proof}[Proof of Theorem \ref{second}]

  Thanks to Proposition \ref{easy}, Corollary \ref{coro} and Proposition \ref{atta}, the theorem is proved once we establish $\mathrm{B}_2\Rightarrow\mathrm{P}_2$.
  Rather than adapting the proof of $\mathrm{B}_1\Rightarrow\mathrm{P}_1$ we prefer to argue
  along the following lines: if $\mathrm{B}_2$ holds, then we show that there exists an invariant set  $E_0\subset E$ (i.e.~$E_0 \in \Ecal$ and $P(x,E_0)=1$ for all $x \in E_0$) of full
  $\mu$-measure on which $\mathrm{B}_1$ holds and hence, by Theorem \ref{first}, $\mathrm{P}_1$ holds. Then we show that  $\mathrm{P}_2$ holds on the full space $E$.

$\mathrm{B}_2\Rightarrow\mathrm{P}_2$:  We are not aware of a simple direct proof that there exists a subset of full $\mu$-measure on which  $\mathrm{B}_1$ holds. Even though
($\mu$-)irreducibility implies that $Q(x,B)=1$ for every $B\in \Ecal$ for which $\mu(B)>0$ and $\mu$-almost every $x \in E$, the exceptional sets may depend on $B$ and there are
(typically) uncountably many such sets $B$.

Since $P$ is irreducible, Proposition \ref{small} shows that there exists a small set $C\in \Ecal$ (with $\nu$ and $m$ as stated there).
We can and will assume that $\nu(E\backslash C)=0$. Define $G:=\{x \in E: Q(x,C)=1\}$. Then $G \in \Ecal$, $G$ is invariant, and $\mu(G)=1$. 
We claim that property $\mathrm{B}_1$ holds on $G$. All we have to show is that $Q(x,B)=1$ for all $x \in G\cap C$ and all $B \in \Ecal$ such that $\mu(B)>0$. Fix such a set $B$
and let $H:=\{x \in G \cap C:\,Q(x,B)=1\}$. Then
$\mu(H)=\mu(C)>0$ and for $x \in H$ we have $P_m(x,H)=P_m(x,C)\ge \nu(C)>0$.  Assume that $y \in G \cap C$ satisfies $Q(y,B)<1$ (i.e.~$y \notin H$). Then,
$P_m(y,H)\ge \nu(H)=\nu(C)$ (since $0=P_m(x,C\backslash H)\ge \nu(C\backslash H)$ for $x \in H$). This means that, whenever the chain is in the set $(C\cap G)\backslash H$, then with
probability at least $\nu(C)>0$ it will hit the set $H$ after $m$ steps. Since the chain starting at $y \in G\cap G$ visits $C\cap G$ infinitely often (almost surely), it follows that
$L(y,H)=1$, contradicting our assumption on $y$. Using Lemma \ref{restrict}, $G$ equipped with the trace $\sigma$-field satisfies our assumption on the state space and we see that
property $\mathrm{B}_1$ holds on $G$.

Theorem \ref{first} shows that property $\mathrm{P}_1$  holds on $G$. Then, clearly, property  $\mathrm{P}_3$  holds on $E$. Since $P$ is irreducible, we have $L(x,G)>0$ and
hence $\lim_{n \to \infty}d(P_n(x,.),\mu)<1$ for every $x \in E$ and therefore  $\mathrm{P}_2$  holds on $E$.
\end{proof}

\begin{proof}[Proof of Theorem \ref{third}] By  Proposition \ref{easy}, Corollary \ref{coro} and Proposition \ref{atta} it suffices to show that $\mathrm{B}_3\Rightarrow\mathrm{P}_3$.

$\mathrm{B}_3\Rightarrow\mathrm{P}_3$: We can argue like in the proof of  $\mathrm{B}_2\Rightarrow\mathrm{P}_2$ (the present argument is even easier). Using the very definition of
weak irreducibility, we find an invariant set $E_0$ of full $\mu$-measure on which $\mathrm{B}_2$ and hence, using Theorem \ref{second}, $\mathrm{P}_2$ hold. Therefore, $\mathrm{P}_3$
holds on $E$.
\end{proof}

\section{Complements, examples, and open problems}\label{Examples}

\begin{remark}\label{uniq}
  We show that Property $\mathrm{P}_2$ implies uniqueness of $\mu$ (as claimed in Remark \ref{12}): assume that $\mu$ and $\tilde \mu$ are
  different ipm's and let $\hat \mu:=\frac 12\big(\mu+\tilde \mu\big)$. Since $\mathrm{P}_2\Leftrightarrow \mathrm{A}_2$ and property $\mathrm{A}_2$ is independent of the chosen ipm,
  we see that $\mathrm{P}_2$ holds with respect to both $\mu$ and $\hat \mu$, so    
  $P_n(x,.)$ converges to $\mu$ for $\mu$-almost all $x$ and to $\hat \mu$ for $\hat \mu$-almost all $x$. Since $\hat \mu\ll \mu$ and $\hat \mu\neq \mu$ this is a contradiction
  (this proof is adapted from \cite[Proof of Corollary 1]{KS15}).
\end{remark}

\begin{ex}\label{peri}
  Let $E:=\{0,1\}$ and $P(0,\{1\})=P(1,\{0\})=1$. Then the unique invariant probability measure $\mu$ is given by $\mu(\{0\})=\mu(\{1\})=1/2$. For this example,
  the second part of property
  P$_2$ holds but the first part doesn't, so the first part of P$_2$ cannot be deleted without changing the class of chains for which P$_2$ holds.
\end{ex}

\begin{ex}\label{simple}
    Let $E:=\N_0$ with the discrete $\sigma$-field $\Ecal$. Define $P(x,\{x-1\})=1$ for $x \ge 2$, $P(1,\{1\})=1$ and $P(0,\{x\})=2^{-x}$ for $x \in \N$.
    Clearly all transition probabilities converge to $\mu=\delta_1$ but $P_n(0,.)$ and $P_n(1,.)$ are non-equivalent for every $n \in \N$ (but the states 0 and 1 are asymptotically
    equivalent).
\end{ex}
  
\begin{ex}\label{standard} (cf.~\cite[Example 5]{KS18}.)
    Let $E:=\N_0$ with the discrete $\sigma$-field $\Ecal$. Define  $P(0,\{0\})=1$ and $P(x,\{x-1\})=1/3$ and $P(x,\{x+1\})=2/3$ for $x \in \N$.
    Clearly, $\mu=\delta_0$ is the unique invariant probability measure and $P_n(x,.)$ does not converge to $\mu$ if $x>0$, so $P$ satisfies $\mathrm{P}_2$ but not $\mathrm{P}_1$.
    Note that for each $x,y \in E$ and $k \ge x\wedge y$, $\zeta:=\delta_0\otimes \delta_0$ satisfies  $\zeta \in \tilde C(P_k(x,.),P_k(y,.))$ and $\zeta(\Delta)=1$, showing that
    if ``$>0$'' in Assumption G$_2$ is replaced by ``$=1$'', then the condition does {\em not} imply $\mathrm{G}_1$.
\end{ex}

\begin{remark}
  Note that Assumption  G$_1$ is formally weaker than requiring that for each pair $(x,y)\in E \times E$ there exists some $\xi \in \tilde  C(\P_x,\P_y)$
  such that $\xi^1 \sim \P_x$ and $\xi^2 \sim \P_y$, but these two conditions are in fact equivalent: according to G$_1$ we find, for each pair $(x,y)$,
  some $\check \xi \in \check C(\P_x,\P_y)$ such that $\lim_{k \to \infty}\check\xi_k(\Delta)=1$ and some $\hat \xi \in \check C(\P_y,\P_x)$ such that $\lim_{k \to \infty}\hat \xi_k(\Delta)=1$.
  Then $\xi:=\frac 12 \check\xi + \frac 12 \hat \xi$ satisfies the formally stronger condition. 
\end{remark}

\begin{remark}
  One may ask whether it is sufficient for P$_1$ to hold if for each pair $(x,y)\in E \times E$ and each $k \in \N_0$ there exists some probability measure $\zeta_k$ on
  $(E \times E,\Ecal\otimes \Ecal)$ whose marginals are equivalent to $P_n(x,.)$ and $P_n(y,.)$ respectively, such that $\lim_{n \to \infty}\zeta_k(\Delta)=1$. Again, Example \ref{standard}
  provides a negative answer. Consider $\xi$ as in the previous example. 
  Then $\lim_{k \to \infty}\xi_k(\Delta)\ge\lim_{k \to \infty}\xi_k(\{(0,0)\}) =1$. Note that the marginals of the measures  $\xi_k$ are equivalent to $P_k(x,.)$ and $P_k(y,.)$
  respectively  but that $\xi^1$ and $\xi^2$ are not equivalent to $\P_x$ respectively $\P_y$.  
\end{remark}

\begin{remark} From Theorem \ref{first} we know that $\mathrm{C}_1\Rightarrow\mathrm{P}_1$ holds since $\mathrm{C}_1\Rightarrow\mathrm{A}_1\Rightarrow\mathrm{B}_1\Rightarrow\mathrm{P}_1$.
  Here we present an essentially  well-known direct proof. For $x \in E$, $n \in \N$, and $A \in \EE$ we have
  \begin{align*}
|\mu(A)-P_n(x,A)|&=\Big|\int_E P_n(y,A)\,\dd \mu(y)-P_n(x,A)\Big|=\Big|\int_E \Big(P_n(y,A)-P_n(x,A)\Big)\, \dd \mu(y)\Big|\\
&\le\int_E \Big| P_n(y,A)-P_n(x,A)\Big| \, \dd \mu(y)\le \int_E d\Big(P_n(y,.), P_n(x,.)\Big) \, \dd \mu(y)
\end{align*}

which converges to 0 by dominated convergence (note that Proposition \ref{mbTV} shows that the last integrand is measurable with respect to $y$), so the claim follows.

In fact, a slight modification of the proof shows the result without employing Proposition \ref{mbTV} (and without assuming that $\Ecal$ is countably generated):

fix $x$ and let $R_n(y,A):=\Big| P_n(y,A)-P_n(x,A)\Big|$, $n \in \N$. There exist sets $A_n \in \Ecal$ such that
$$
U_n:=\sup_{A \in \Ecal} \Big(\int_E R_n(y,A)\,\dd \mu(y)\Big)\le \int_E R_n(y,A_n)\,\dd \mu(y) +2^{-n},
$$
which converges to 0 as $n \to \infty$ by dominated convergence.
\end{remark}

\begin{remark}\label{allgemein}
  It seems to be an open question whether all properties stated in Theorem \ref{first} are equivalent even in the case in which $(E,\Ecal)$ is not Borel (and similarly for Theorems
  \ref{second} and \ref{third}). The present proof which is based on the gluing lemma \ref{glu} can not be applied in this case:  \cite{BPR15} contains an example of a
  separable and metric space equipped with its Borel $\sigma$-field for which the conclusion in the gluing lemma fails. 
\end{remark}

\appendix

\section{Auxiliary results and measurability issues}

\subsection{$\mu$-irreducibility and the existence of small sets}

We start with a proposition which was announced in Remark \ref{rem:irr} and whose proof is inspired by that of \cite[Proposition 4.2.2]{MT09}.

\begin{prop}\label{irredu}
If $P$ is $\phi$-irreducible, then $P$ is $\mu$-irreducible.
\end{prop}

\begin{proof}
  Let $P$ be $\phi$-irreducible. Then $\phi \ll \mu$ (see Remark \ref{rem:irr}) and, due to Lebesgue's theorem, there exists a
  set $B \in \Ecal$ such that $\phi$ and $\mu$ restricted to $B$ are equivalent and $\phi(B^c)=0$.
  Note that $\mu(B)>0$. If $\mu(B^c)=0$, then $\phi \sim \mu$ and we are done, so we assume that $\mu(B^c)>0$. We
  have to show that for any measurable set $C \subset B^c$ such that $\mu(C)>0$ we have $L(x,C)>0$ for every $x \in E$. Fix such
  $x$ and $C$ and define the measure
$$
\nu(.):=\int_B\sum_{m=1}^\infty 2^{-m}P_m(y,.)\,\dd \mu(y).
$$
Invariance of $\mu$ implies $\nu \ll \mu$ and that the restriction of both measures to $B$ are equivalent.
Let $G \in \Ecal$ be a set such that $\nu \sim \mu$ on $G$, $\nu(G^c)=0$ and $B \subset G$.

First, we assume that $\mu(G^c)>0$. Let $m_0\in \N$ be such that
$\int_{G^c}P_{m_0}(y,G)\,\dd \mu(y)>0$ (such an $m_0$ exists since $P$ is $\phi$-irreducible). Using invariance of  $\mu$,
we obtain
$$
\int_GP_{m_0}(y,G^c)\,\dd \mu(y)=\int_{G^c}P_{m_0}(y,G)\,\dd \mu(y)>0.
$$

Therefore, there exists some $\varepsilon_1>0$ such that for  $D:=\{y \in G:\,P_{m_0}(y,G^c)\ge \varepsilon_1\}$, we have $\mu(D)>0$ and hence
$\nu(D)>0$, which means that there exists some $m_1\in \N$ such that $\int_B P_{m_1}(y,D)\,\dd \mu(y)>0$. 

Therefore,
\begin{align*}
  \nu(G^c)&\ge \int_B 2^{-m_0-m_1}  P_{m_0+m_1}(y,G^c)\,\dd \mu(y)\\
          &\ge 2^{-m_0-m_1}  \int_B \int_D P_{m_0}(z,G^c)  P_{m_1}(y,\dd z)\,\dd \mu(y)\\
          &\ge 2^{-m_0-m_1}  \varepsilon_1 \int_B P_{m_1}(y,D)\,\dd \mu(y)>0,
\end{align*}
contradicting the definition of $G$, so $\mu(G^c)=0$.

In this  case $\mu \sim \nu$ and so $\nu(C)>0$ which implies that there exist some $\varepsilon_2>0$ and $m_2 \in \N$
such that $\tilde D:=\{y \in B:\,P_{m_2}(y,C)\ge \varepsilon_2\}$ satisfies $\mu(\tilde D)>0$. $\phi$-irreducibility and the
definition of the set $B$ imply $L(x,\tilde D)>0$, which, together with the definition of $\tilde D$, implies $L(x,C)>0$,
so the proof of the proposition is complete.
\end{proof}

The following proposition is an easy consequence of the rather deep Theorem 5.2.2 in \cite{MT09} (which is a key step in the proof of  $\mathrm{B_1}\Rightarrow \mathrm{P_1}$
(in our notation)) and of the (not so deep) previous proposition. 

\begin{prop}\label{small}(\cite[Theorem 5.2.2]{MT09}) 
  Let $P$ be irreducible. Then there exists a  {\em small set} $C$, i.e.~a set $C \in \Ecal$ such that $\mu(C)>0$
  for which there exist a finite measure $\nu$ and some $m \in \N$ such that $\nu(C)>0$ and $P_m(x,B)\ge \nu(B)$ for all $x \in C$ and $B \in \Ecal$.
\end{prop}

\begin{proof}
  Theorem 5.2.2 in \cite{MT09} assumes that $P$ is $\psi$-irreducible where $\psi$ is a {\em maximal irreducibility measure}. By the previous proposition we can take
  $\psi=\mu$  and therefore the  conclusions of  \cite[Theorem 5.2.2]{MT09} and  of Proposition \ref{small} are the same.
\end{proof}

\subsection{A gluing lemma}
A proof of the following {\em gluing lemma} can be found in \cite[Lemma 4.]{BPR15} (or in \cite[Lemma 4.3.2]{K18} under the additional condition that the spaces are standard Borel).
The conditions in  \cite[Lemma 4.]{BPR15} are even slightly weaker than ours. 
\begin{lemma}\label{glu}
  Let $(E_i,\EE_i)$, $i=1,2,3$ be Borel spaces and let $\rho_1$ and $\rho_3$ be probability measures on $(E_1\times E_2,\EE_1 \otimes \EE_2)$ and $(E_2\times E_3,\EE_2 \otimes \EE_3)$ respectively such that $\rho_1(E_1 \times B)=\rho_3(B \times E_3)$ for all $B \in \EE_2$. Then there exists a probability measure $\mu$ on
  $(E_1\times E_2\times E_3,\EE_1 \otimes \EE_2\otimes \EE_3)$ such that $\mu(A\times E_3)=\rho_1(A)$ for all $A \in  \EE_1 \otimes \EE_2$ and $\mu(E_1 \times B)=\rho_3(B)$
  for all $B \in \EE_2 \otimes \EE_3$.
\end{lemma}

\subsection{Measurability issues}

\begin{lemma}\label{restrict}
  Let $\tilde E \in \Ecal$ satisfy $\mu(\tilde E)=1$. Then there exists a set $\hat E \subset \tilde E$ in $\Ecal$ such that $P(x,\hat E)=1$ for all $x \in \hat E$ and
  $\mu(\hat E)=1$. Further, for any $\tilde E \in \Ecal$, $\tilde E$ equipped with the trace $\sigma$-field of $\Ecal$ satisfies our basic assumptions (countably generated
  $\sigma$-field  and measurable diagonal).
\end{lemma}

\begin{proof}
    The last statement is clear. To see the first, define $E_0:=\tilde E$ and $E_{i+1}:=\{x \in E_i:\,P(x,E_i)=1\}$, $i \in \N_0$. Then  $\hat E:=\bigcap_i E_i$ does the job.
\end{proof}

  In the following two  statements we assume that $(E,\Ecal)$ satisfies our general assumptions spelled out in the introduction and that
  $Q$ and $\tilde Q$ are Markov kernels on $E$. 
 
  \begin{lemma}\cite[p. 30f.]{K18}\label{mbRN} 
    Let $\Lambda(x,y;\dd z):=\frac 12 \big( Q(x,\dd z)+\tilde Q(y,\dd z)\big)$. There exist  measurable maps $f$ and $\tilde f$ such that for each
    $A\in \Ecal$, we have
   $$
Q(x,A)=\int_Af(x,y;z)\,\Lambda(x,y,\dd z),\quad \tilde Q(y,A)=\int_A\tilde f(x,y;z)\,\Lambda(x,y,\dd z).
    $$
    \end{lemma}

    This lemma is used in \cite{K18} to prove a result which, in particular, implies the following proposition (which is not immediate since
  the supremum of an uncountable family of real-valued measurable functions need not be measurable). 
  \begin{prop}\cite[Theorem 2.2.4 (i)]{K18}\label{mbTV}
The function
    $$
      (x,y) \mapsto d\big(Q(x,.),\tilde Q(y,.)\big)
    $$
is measurable.
\end{prop}

\begin{lemma}\label{mbae}
The set of all $(x,y)\in E \times E$ for which $x$ and $y$ are asymptotically equivalent is a measurable subset of $(E \times E,\Ecal \otimes \Ecal)$.
  \end{lemma}

\begin{proof}
 
  Applying Lemma \ref{mbRN} with $Q=\tilde Q=P_n$ we see that there exists a jointly measurable function $f_n$ such that
  $$
P_n(x,A)=\int_A f_n(x,y;z)\Lambda_n(x,y;\dd z),\;P_n(y,A)=\int_A f_n(y,x;z)\Lambda_n(x,y;\dd z),
$$
for all $x,y \in E$ (with $\Lambda_n$ defined as in Lemma \ref{mbRN}). Defining $A_n(x,y):=\{z \in E: \,f_n(x,y;z)f_n(y,x;z)>0\}$, we see that $A_n(x,y)\in \Ecal$ and that
$P_n(x,.)$ and $P_n(y,.)$ restricted to $A_n(x,y)$ are equivalent. Further, $A_n(x,y)$ is the largest set (up to sets of measure 0 with respect to
$\Lambda_n(x,y;.)$) with this property. Observe that the map $(x,y)\mapsto P_n(x, A_n(x,y))=\int 1_{A_n(x,y)}(z)\,P_n(x,\dd z)$ is measurable (by a well-known application of the monotone class theorem) since the integrand is jointly measurable. The claim follows since $x$ and $y$ are asymptotically equivalent iff $\lim_{n \to \infty}P_n(x,A_n(x,y))=\lim_{n \to \infty}P_n(y,A_n(x,y))=1$. 
\end{proof}

\end{document}